\documentclass[11pt]{amsart}
\usepackage{amsmath}%
\usepackage{amsfonts}%
\usepackage{amssymb}%
\usepackage{graphicx}%
\usepackage{mathtools}%
\usepackage{xcolor}%
\usepackage{nameref,hyperref,url}
\usepackage[top=1in, bottom=1.25in, left=1in, right=1in]{geometry}
\usepackage[foot]{amsaddr}
\newtheorem{theorem}{Theorem}
\theoremstyle{plain}

\newtheorem{definition}{Definition}

\newtheorem{lemma}{Lemma}

\newtheorem{proposition}{Proposition}
\newtheorem{remark}{Remark}

\numberwithin{equation}{section}
\newcommand{\vertiii}[1]{{\left\vert\kern-0.25ex\left\vert\kern-0.25ex\left\vert #1 
    \right\vert\kern-0.25ex\right\vert\kern-0.25ex\right\vert}}
\newcommand{\vertiiin}[1]{{\vert\kern-0.25ex\vert\kern-0.25ex\vert #1 
    \vert\kern-0.25ex\vert\kern-0.25ex\vert}}

\begin{document}
\title[Pervasiveness of the $p\,$-Laplace operator under localization]{Pervasiveness of the $p\,$-Laplace operator under localization of fractional $g\,$-Laplace operators}

\author{Alejandro Ortega}
\email[A. Ortega ]{alortega@math.uc3m.es}
\address[A. Ortega]{Departamento de Matem\'aticas,  
Universidad Carlos III de Madrid, Av. Universidad 30, 28911 Legan\'es (Madrid), Spain}

\thanks{The author is partially supported by the State Research Agency of Spain, under research project PID2019-106122GB-I00.}

\begin{abstract}
In this work we analyze the behavior of truncated functionals as
\begin{equation*}
\int_{\mathbb{R}^N}\int_{B(x,\delta)} G\left(\frac{|u(x)-u(y)|}{|x-y|^{s}}\right)\frac{dydx}{|x-y|^N}\qquad\text{for }\delta\to0^+.
\end{equation*}
Here the function $G$ is an Orlicz function that in addition is assumed to be a regularly varying function at $0$. A prototype of such function is given by $G(t)=t^p(1+|\log(t)|)$ with $p\geq2$. These kind of functionals arise naturally in {\it peridynamics}, where long-range interactions are neglected and only those exerted at distance smaller than $\delta>0$ are taken into account, i.e., the {\it horizon} $\delta>0$ represents the range of interactions or nonlocality.\\ This work is inspired by the celebrated result by Bourgain, Brezis and Mironescu, who analyzed the limit $s\to1^-$ with $G(t)=t^p$. In particular, we prove that, under appropriate conditions, 
\begin{equation*}
\lim\limits_{\delta\to0^+}\frac{p(1-s)}{G(\delta^{1-s})}\int_{\mathbb{R}^N}\int_{B(x,\delta)}G\left(\frac{|u(x)-u(y)|}{|x-y|^{s}}\right)\frac{dydx}{|x-y|^N}=K_{N,p}\int_{\mathbb{R}^N}|\nabla u(x)|^p dx,
\end{equation*}
for $p=index(G)$ and an explicit constant $K_{N,p}>0$. Moreover, the converse is also true, if the above localization limit exist as $\delta\to0^+$, the Orlicz function $G$ is a regularly varying function with $index(G)=p$.
\end{abstract}
\maketitle


\section{Introduction}
In the celebrated paper by Bourgain, Brezis and Mironescu \cite{Bourgain2001}, the authors proved the following nowadays well-known convergence result.
\begin{theorem}\label{ThBBM} 
Given $u\in L^p(\mathbb{R}^N)$ and $0<s<1$, it holds that
\begin{equation*}
\lim\limits_{s\to1^-}(1-s)\int_{\mathbb{R}^N}\int_{\mathbb{R}^N}\frac{|u(x)-u(y)|^p}{|x-y|^{N+sp}}dydx=K_{N,p}\int_{\mathbb{R}^N}|\nabla u(x)|^pdx,
\end{equation*}
with the convention that $\int_{\mathbb{R}^N}|\nabla u|^pdx=\infty$ if $u\notin W^{1,p}(\mathbb{R}^N)$ and, for $e$ unitary vector,
\begin{equation*}
K_{N,p}=\frac{1}{|\mathbb{S}^{N-1}|}\int_{\mathbb{S}^{N-1}}|w\cdot e|^p d\sigma_{w}=\frac{1}{\sqrt{\pi}}\frac{\Gamma\left(\frac{N}{2}\right)\Gamma\left(\frac{p+1}{2}\right)}{\Gamma\left(\frac{N+p}{2}\right)}.
\end{equation*}
\end{theorem}
This work led to the development of an extensive literature concerning convergence of functionals in the sense of $\Gamma$-convergence (cf. \cite{Alberico2020, Bal2020, Brezis2016, Ferreira2020, Leoni2011, Leoni2014, Mazcprimeya2002, Mazcprimeya2003, Nguyen2006, Ponce2004}  and references therein) as well as recently advances in the convergence of the spectrum of related operators (cf. \cite{Bellido2021,Bellido2021a, Brasco2016,FernandezBonder2019, Salort2020, Salort2022, Salort2023}). In particular, in \cite{FernandezBonder2019}, the following is proved (we refer to Section \ref{Section:Functional} for the precise definitions).
\begin{theorem}\label{lim_s} 
Let $G$ be an Orlicz function such that the following limit exists,
\begin{equation*}
\tilde{G}(a)=\lim\limits_{s\to1^+}(1-s)\int_0^1\int_{\mathbb{S}^{N-1}}G(a|z_n|r^{1-s})dS_z\frac{dr}{r}.
\end{equation*}
Then, given $u\in L^G(\mathbb{R}^N)$ and $0<s<1$, it holds that
\begin{equation*}
\lim\limits_{s\to1^-}(1-s)\int_{\mathbb{R}^N}\int_{\mathbb{R}^N}G\left(\frac{|u(x)-u(y)|}{|x-y|^{s}}\right)\frac{dydx}{|x-y|^N}=\int_{\mathbb{R}^N}\tilde{G}(|\nabla u(x)|)dx.
\end{equation*}
\end{theorem}
The authors also obtain a full $\Gamma$-convergence result used to deduce the convergence of solutions as $s\to 1^-$ for some fractional versions of the $g\,$-Laplace operator,
\begin{equation*}
\Delta_g=div\left(g(|\nabla u|)\frac{\nabla u}{|\nabla u|}\right).
\end{equation*}
The aim of this work is to analyze the behavior of these functionals under localization. Precisely, given an {\it horizon} $\delta>0$ and $0<s<1$, we analyze the behavior as $\delta\to0^+$ of functionals of the form
\begin{equation}\label{genfunct}
\int_{\mathbb{R}^N}\int_{B(x,\delta)} G\left(\frac{|u(x)-u(y)|}{|x-y|^{s}}\right)\frac{dydx}{|x-y|^N}.
\end{equation}
This type of problem arises naturally in {\it peridynamics} and the associated nonlocal operator, namely
\begin{equation*}
(-\Delta_g)_{\delta}^s u(x)= p.v. \int_{B(x,\delta)}g(|D_su|)\frac{D_su}{|D_su|}\frac{dy}{|x-y|^N},
\end{equation*} 
with $g(t)=G'(t)$ and
\begin{equation*}
D_su=D_su(x,y)=\frac{u(x)-u(y)}{|x-y|^s},
\end{equation*}
can be actually seen as a \textit{ peridynamic fractional $g$-Laplacian}. Peridynamics is a nonlocal continuum model for Solid Mechanics proposed by Silling in \cite{Silling2000}. The main difference between classical theory and peridynamics is based on nonlocality, which refers to the phenomenon where points that are separated up to a positive distance exert a force on each other. This property distinguishes peridynamics from classical theories that rely on gradients. As a result, peridynamics is well-suited for problems involving discontinuities, such as fracture, dislocation, or multi-scale materials. The operator $(-\Delta_g)_\delta^s$ corresponds to a truncation of the \textit{fractional $g\,$-Laplacian} analyzed in \cite{Bahrouni2022,FernandezBonder2019,Salort2023,Salort2022},
\begin{equation*}
(-\Delta_g)^s u(x)= p.v. \int_{\mathbb{R}^N}g(|D_su|)\frac{D_su}{|D_su|}\frac{dy}{|x-y|^N}.
\end{equation*} 
As we are interested in the limit behavior as $\delta\to 0^+$, we consider Orlicz functions $G$ that, in addition, are supposed to belong to the class of regularly varying functions at 0, say $\mathcal{RV}_{\rho}(0)$ for some positive $\rho\in\mathbb{R}$. Roughly speaking, a function $G(t)$ is said to be a regularly varying function  at $0$ if the following limit 
\begin{equation}\label{eq:karamata}
\lim\limits_{t\to0^+}\frac{G(\lambda t)}{G(t)}=h_G(\lambda),\quad\text{exists for all }\lambda>0.
\end{equation}
An example of an Orlicz function belonging to the class $\mathcal{RV}_{p}(0)$ is the function
\begin{equation}\label{prototype}
G(t)=t^p(1+|\log (t)|).
\end{equation}
A complete characterization of the functions for which the limit \eqref{eq:karamata} exists was provided by J. Karamata (cf. \cite{Karamata1931}), leading to the beginning of the theory of Regularly Varying functions. We will use such characterization (see Theorem \ref{teo:karamata} below) to prove the main result of this work, which is introduced next.
\begin{theorem}\label{mainTheorem}
Let $G$ be an Orlicz function satisfying \eqref{hypotheses}. Then, given $u\in L^G(\mathbb{R}^N)$ and $0<s<1$, it holds that
\begin{equation*}
\lim\limits_{\delta\to0^+}\frac{p(1-s)}{G(\delta^{1-s})}\int_{\mathbb{R}^N}\int_{B(x,\delta)}G\left(\frac{|u(x)-u(y)|}{|x-y|^{s}}\right)\frac{dydx}{|x-y|^N}=K_{N,p}\int_{\mathbb{R}^N}|\nabla u(x)|^p dx,
\end{equation*}
where $p=index(G)$ and with the convention that $\int_{\mathbb{R}^N}|\nabla u|^pdx=\infty$ if $u\notin W^{1,p}(\mathbb{R}^N)$.
\end{theorem}
We also prove the following.
\begin{theorem}\label{Th2}
Let $G$ be an Orlicz function satisfying \eqref{hypotheses} and $u\in L^G(\mathbb{R}^N)$. If 
\begin{equation}\label{hyp_inf}
\liminf\limits_{\delta\to 0^+}\frac{1}{G(\delta^{1-s})}\int_{\mathbb{R}^N}\int_{B(x,\delta)}G\left(\frac{|u(x)-u(y)|}{|x-y|^{s}}\right)\frac{dydx}{|x-y|^N}<\infty,
\end{equation}
then $u\in W^{1,p}(\mathbb{R}^N)$. 
\end{theorem}
The above Theorem \ref{mainTheorem} is closely related to the next localization result dealing with the fractional $p\,$-Laplacian (cf. \cite{Bellido2021a}). Set $\Omega_{\delta}=\{z\in\mathbb{R}^N: |x-z|<\delta,\ \text{for }x\in\Omega\}$.
\begin{theorem}\label{BellidoOrtega}
Let $p>1$, $0<s<1$ and $u\in L^p(\mathbb{R}^N)$. Then,
\begin{equation*}
\lim\limits_{\delta\to0^+}\frac{p(1-s)}{\delta^{p(1-s)}}\int_{\Omega_\delta}\int_{\Omega_\delta\cap B(x,\delta)}\frac{|u(x)-u(y)|^p}{|x-y|^{N+sp}}dydx=K_{N,p}\int_{\Omega}|\nabla u|^pdx,
\end{equation*}
with the convention that $\int_{\mathbb{R}^N}|\nabla u(x)|^pdx=\infty$ if $u\notin W^{1,p}(\mathbb{R}^N)$.
\end{theorem}
The proof of Theorem \ref{BellidoOrtega} relies on a general $\Gamma$-convergence result (cf. \cite[Theorem 1]{Bellido2015}) which in turns uses the homogeneity of the potential function in a crucial way. In particular, translated into our setting, \cite[Theorem 1]{Bellido2015} requires the function $G(t)$ to be close to an homogeneous function as $t\to0^+$, namely, given $a\in\mathbb{R}_+$ it requires the existence of the limit
\begin{equation*}
\lim\limits_{t\to0^+}\frac{1}{t^\beta}G\left(at^{1-s}\right)=G^{\circ}(a),
\end{equation*} 
for some $\beta\in\mathbb{R}$ and some function $G^\circ(a)$. Obviously, a function like \eqref{prototype} is not comparable to any power $t^q$, $\beta>1$ in the sense that
\begin{equation*}
\lim\limits_{t\to0^+}\frac{G(t)}{t^\beta}=\left\{\begin{tabular}{rrl} $0$&& if $1<\beta< p$,\\
$+\infty$&& if $\beta\geq p$.
\end{tabular}\right.
\end{equation*} 
It is also worth to note the following. If $G(t)=t^p$, then the function defined in Theorem \ref{lim_s} is given by $\tilde{G}(t)=\frac{K_{N,p}}{p}\,t^p$ and hence, the limit $s\to1^-$ produces, under the appropriate scaling, the same result as the localization produced by the limit $\delta\to0^+$ appropriately scaled. This is no longer true for a general function $G(t)=t^p\ell(t)$. Actually, due to the characterization of the Regularly Varying functions (see Theorem \ref{teo:karamata} below), if the limit \eqref{eq:karamata} exists then the function $G$ has to be a Regularly Varying function and the function $h_G(t)$ is necessarily a power function $h_G(t)=t^p$, while the limit function appearing in Theorem \ref{lim_s} is in general different from a power function (see Remark \ref{remark}).

Nonlocal functionals involving Orlicz functions has been also recently studied in \cite{Correa2020}. In particular, the authors consider functionals of the form
\begin{equation}\label{correa}
\int_{\mathbb{R}^N}\int_{\mathbb{R}^N} G(|u(x)-u(y)|)J(|x-y|)dydx,
\end{equation}
for an Orlicz function $G$ and an appropriate kernel $J$. Theorem \ref{mainTheorem} can be easily adapted to functionals like \eqref{correa} and if, for instance, $J(z)=|z|^{-(N+ps)}$, then we have 
 \begin{equation*}
\lim\limits_{\delta\to0^+}\frac{\delta^{sp}}{G(\delta)}\int_{\mathbb{R}^N}\int_{B(x,\delta)} G(|u(x)-u(y)|)J(|x-y|)dydx=\frac{K_{N,p}}{p(1-s)}\int_{\mathbb{R}^N}|\nabla u(x)|^p dx.
\end{equation*}
In general, if $G$ is a regularly varying function of index $p>1$ and the kernel $J$ is also assumed to be a regularly varying function of index $q\in\mathbb{R}$, $0<N+p+q$, then (see Proposition \ref{prop:correa}),
 \begin{equation*}
\lim\limits_{\delta\to0^+}\frac{N+p+q}{G(\delta)J(\delta)\delta^N}\int_{\mathbb{R}^N}\int_{B(x,\delta)} G(|u(x)-u(y)|)J(|x-y|)dydx=K_{N,p}\int_{\mathbb{R}^N}|\nabla u(x)|^p dx.
\end{equation*}
By its very definition, functionals like \eqref{genfunct} only take into account the contribution in a neighborhood of a given point $x\in\mathbb{R}^N$ neglecting the tails of the function $u$. This phenomena entails some difficulties when passing to the limit $\delta\to0^+$, since, contrary to the case $s\to1^-$, where the integral is always considered in the whole $\mathbb{R}^N$, the localization of a function at a point $x_0\in\mathbb{R}^N$, by its very nature, losses the information of the behavior of $u$ far away from $x_0$ (see Section \ref{Section:Further}). 

\vspace{0.5cm}

\textbf{Organization of the paper:} In Section \ref{Section:Functional} we recall some important properties of Orlicz functions as well as for Regularly Varying functions. We finish this section by stating some classical facts about the Orlicz-Sobolev spaces. Section \ref{Section:Technical} contains some technical results used to prove Theorem \ref{mainTheorem} whose proof is contained in Section \ref{Section:Asym}. Finally, Section \ref{Section:Further} is devoted to some comments about the $\Gamma$-convergence of functionals like \eqref{genfunct} and its application to the study of the behavior of the spectrum of $(-\Delta_g)_{\delta}^s$ under the limit $\delta\to0^+$.
\section{Functional Setting}\label{Section:Functional}

\subsection{Orlicz and Regularly Varying functions}\hfill\newline
In this section we introduce the Orlicz and the Regularly Varying functions and we recall some of its properties that play a crucial role in the proof of Theorem \ref{mainTheorem}. We also present the fractional Orlicz-Sobolev spaces and some basic properties recently developed in \cite{FernandezBonder2019}.
\subsubsection{Orlicz functions}
\begin{definition}\label{def:orlicz}
A function $G:\mathbb{R}_+\mapsto\mathbb{R}_+$ is said to be an Orlicz function if 
\begin{itemize}
\item[$H_1$)] $G$ is continuous, convex, increasing and $G(0)=0$. 
\item[$H_2$)] $G$ satisfies the $\Delta_2$ condition, i.e., there exists $\mathfrak{c}>2$ such that 
\begin{equation*}
G(2t)\leq \mathfrak{c}\, G(t),\qquad\text{for all }t\in\mathbb{R}_+.
\end{equation*} 
\item[$H_3$)] $G$ is super-linear at $0$, i.e., $\displaystyle \lim\limits_{t\to0^+}\frac{G(t)}{t}=0$.
\end{itemize}
\end{definition}
Some examples of Orlicz functions are (cf. \cite{Krasn, FernandezBonder2019}),
\begin{itemize}
\item $G(t)=t^p$ with $p>1$.
\item $G(t)=t^p(1+|\log (t)|)$ with $p\geq2$.
\item In general $G(t)=t^p(1+|\log (t)|)^q$ with $p\geq2$ and $0\leq q\leq 1$.
\end{itemize}
\begin{lemma}[{\cite[Lemma 2.4]{FernandezBonder2019}}]
Let $G:\mathbb{R}_+\mapsto\mathbb{R}_+$ be an Orlicz function, then
\begin{itemize}
\item[$P_1$)] $G$ is Lipschitz continuous.
\item[$P_2$)] Given $s\in(0,1)$, the function $G$ is integrability near 0 and infinity, namely,
\begin{equation*}
\int_{1}^{\infty}\frac{G(x^{-s})}{x}\, dx\leq\frac{\mathfrak{g}}{s}\qquad\text{and}\qquad \int_{0}^{1}\frac{G(x^{1-s})}{x}\, dx\leq\frac{\mathfrak{g}}{1-s},
\end{equation*} 
where $\mathfrak{g}=\sup\limits_{x\in(0,1)}x^{-1}G(x)$.
\item[$P_3$)] $G(t)=\int_0^tg(s)ds$ for a non-decreasing right continuous function $g(t)$.
\item[$P_4$)] $G$ is subadditive, namely
\begin{equation*}
G(a+b)\leq \frac{\mathfrak{c}}{2}(G(a)+G(b)),\qquad\text{for all }a,b\in\mathbb{R}_+.
\end{equation*}
\item[$P_5$)] For any $a\geq 0$ it holds that $G(ab)\leq bG(a)$ for $0<b<1$
\end{itemize}
\end{lemma}
Then $\Delta_2$ condition is equivalent to  the inequality (cf. \cite{Krasn}),
\begin{equation*}
\frac{G'(t)}{G(t)}\leq\frac{p}{t},\qquad\text{for all } t>0, 
\end{equation*} 
for some $p>1$. Hence, (cf. \cite[Lemma 2.5]{FernandezBonder2019}), for every $t\geq0$ and $\lambda\geq1$,
\begin{equation*}
G(\lambda t)\leq \lambda^p G(t).
\end{equation*}
In addition (cf. \cite[Lemma 2.7]{FernandezBonder2019}) there exists $q>1$ such that, for all $t\geq0$ and $0\leq \lambda\leq 1$,
\begin{equation*}
\lambda^{2q}G(t)\leq G(\lambda t).
\end{equation*} 
Summarizing, the function $G$ satisfies the growth condition
\begin{equation}\label{pbounds}
p^-\leq \frac{t G'(t)}{G(t)}\leq p^+\qquad\text{for all }t> 0,
\end{equation}
and some $0<p^-\leq p^+$. Thus, 
\begin{equation}\label{comparison}
\min\{\lambda^{p^-},\lambda^{p^+}\}G(t)\leq G(\lambda t)\leq\max\{\lambda^{p-},\lambda^{p^+}\}G(t),\qquad\text{for all }\lambda,t\geq 0.
\end{equation}
For instance,
\begin{itemize}
\item If $G(t)=t^p$ with $p>1$, then $p^{\pm}=p$,
\item If $G(t)=t^p(1+|\log(t)|)$, then $p^{\pm}=p\pm 1$.
\end{itemize}
As a consequence the function $G$ also satisfies de so-called $\nabla_2$-condition, namely
\begin{equation}\label{nabla_condition}
G(2t)\geq c G(t),\qquad\text{for all }t\geq 0
\end{equation}
for a positive constant $c>2$. Actually, because of \eqref{comparison},  we have $c=2^{p^-}$.

We finish this section by recalling the H\"older-like conjugate for these Orlicz spaces. Let us set the complementary function $G^*$ be defined as
\begin{equation*}
G^*(a)=\sup\limits_{t>0}\{at-G(t)\}.
\end{equation*}
It is easy to see that, for $G(t)=\frac{t^p}{p}$ then $G^*(t)=\frac{t^q}{q}$ with $\frac{1}{p}+\frac{1}{q}=1$, i.e., the standard H\"older conjugate. By definition of $G^*$, the following Young-type inequality holds
\begin{equation*}
at\leq G(t)+G^*(a),\qquad\text{for all }a,t\geq 0.
\end{equation*}
\subsubsection{Regularly Varying functions}\hfill\newline
Let $G$ be a positive function defined on some neighborhood $[a,+\infty)$ of infinity and assume that, for all $\lambda>0$,
\begin{equation*}
\frac{G(\lambda t)}{G(t)}\to h(\lambda),\qquad\text{as }t\to+\infty
\end{equation*}
If $G$ has some minimal smoothness, for instance measurability, the above convergence holds uniformly on compact subsets of $(0,+\infty)$. Furthermore, the function $h$ is necessarily a power function. Investigating these properties and similar relationships, along with their diverse applications, forms the basis of the theory of functions of regular variation introduced by J. Karamata (cf. \cite{Karamata1931,Karamata1931a,Karamata1933}). 
\begin{definition}
A function $G:\mathbb{R}_+\mapsto\mathbb{R}_+$ is said to be a regularly function at $+\infty$ if, for all $\lambda>0$,
\begin{equation*}
\lim\limits_{t\to+\infty}\frac{G(\lambda t)}{G(t)}=h_G(\lambda)\in\mathbb{R}_+
\end{equation*} 
and the limit is finite.
\end{definition}
Some key results on the theory of regularly varying functions, and fundamental along this work, are the characterization results by Karamata. 
\begin{theorem}[\cite{Karamata1933}]\label{teo:karamata}
Every regularly varying function $G:\mathbb{R}_+\mapsto\mathbb{R}_+$ is of the form
\begin{equation*}
G(t)=t^{\beta}\ell(t),
\end{equation*}
where $\beta\in\mathbb{R}$ is called the {\it index} of $G$ and $\ell:\mathbb{R}_+\mapsto\mathbb{R}_+$ is a {\it slowly varying} function,
\begin{equation*}
\lim\limits_{t\to+\infty}\frac{\ell(\lambda t)}{\ell(t)}=1.
\end{equation*}
\end{theorem}
Accordingly,
\begin{definition}
A function $G:\mathbb{R}_+\mapsto\mathbb{R}_+$ is said to be a regularly varying function at $+\infty$ of index $p\in\mathbb{R}$, say $G\in\mathcal{RV}_p(+\infty)$, if and only if 
\begin{equation*}
G(t)=t^{p}\ell(t),
\end{equation*}
for some slowly varying function $\ell: \mathbb{R}_+\mapsto\mathbb{R}_+$.
\end{definition}
Note that an slowly varying function is a regularly varying function of index 0. If we transfer our attention from infinity to the origin; thus if $G(t)$ is measurable, positive and
\begin{equation*}
\lim\limits_{t\to0^+}\frac{G(\lambda t)}{G(t)}=\lambda^p,\qquad\text{for all }\lambda>0,
\end{equation*}
we say that $G$ is regularly varying (on the right) at the origin with index $p\in\mathbb{R}$, namely $G\in\mathcal{RV}_p(0)$. Plainly, this is equivalent to (cf. \cite{Krasn}),
\begin{equation}\label{origintoinfty}
G(t)\in\mathcal{RV}_{p}(0)\quad \Longleftrightarrow\quad G\left(\frac{1}{t}\right)\in\mathcal{RV}_{-p}(+\infty).
\end{equation}
Among the fundamental results proven for regularly varying functions we will need the following.
\begin{theorem}[{Uniform Convergence Theorem \cite{Krasn}}]\label{unifth}
Let a continuous function $G\in\mathcal{RV}_p(+\infty)$ (if $p>0$, then assume that $G$ is bounded on each interval $(0,a]$), then
\begin{equation*}
\lim\limits_{t\to+\infty}\frac{G(\lambda t)}{G(t)}=\lambda^p\quad\text{uniformly in $\lambda$ on each}\quad 
\left\{
\begin{tabular}{ll}
 $[a,b],\ 0<a\leq b<+\infty$ & if $p=0$,\\
$(0,b],\ 0< b<+\infty$ & if $p>0$,\\
 $[a,+\infty),\ 0< a<+\infty$ & if $p<0$.
\end{tabular}
\right.
\end{equation*}
\end{theorem}
When dealing with the limit $\delta\to0^+$ of the scaled functional, the next characterization will play also a curcial role.
\begin{theorem}[Karamata. Direct part]\label{KaramataDirect}
Let $G\in\mathcal{RV}_p(+\infty)$ be locally bounded on $[a,\infty)$. Then,
\begin{itemize}
\item for $\sigma\geq-(p+1)$, we have
\begin{equation*}
\frac{t^{\sigma+1}G(t)}{\displaystyle\int_a^t s^\sigma G(s)ds}\to\sigma+p+1,\qquad\text{as }t\to+\infty,
\end{equation*}
\item for $\sigma<-(p+1)$, we have
\begin{equation*}
\frac{t^{\sigma+1}G(t)}{\displaystyle\int_t^\infty s^\sigma G(s)ds}\to-(\sigma+p+1),\qquad\text{as }t\to+\infty.
\end{equation*}
\end{itemize}
\end{theorem}
\begin{theorem}[Karamata. Converse part]\label{KaramataConverse}
Let $G$ be a positive, measurable, locally integrable function on $[a,\infty)$ with $a>0$. Then,
\begin{itemize}
\item if, for some $\sigma>-(p+1)$, we have
\begin{equation*}
\frac{t^{\sigma+1}G(t)}{\displaystyle\int_a^t s^\sigma G(s)ds}\to\sigma+p+1,\qquad\text{as }t\to+\infty,
\end{equation*}
then $G\in\mathcal{RV}_p(+\infty)$.
\item if, for some $\sigma<-(p+1)$, we have
\begin{equation*}
\frac{t^{\sigma+1}G(t)}{\displaystyle\int_t^\infty s^\sigma G(s)ds}\to-(\sigma+p+1),\qquad\text{as }t\to+\infty,
\end{equation*}
then $G\in\mathcal{RV}_p(+\infty)$.
\end{itemize}
\end{theorem}
Because of relation \eqref{origintoinfty}, Theorems \ref{unifth}, \ref{KaramataDirect}, and \ref{KaramataConverse} together with the fact $G(0)=0$, we have (cf. \cite[Theorem 1.7.2b]{Krasn}),
\begin{equation}\label{unifconvergence}
\lim\limits_{t\to+\infty}\frac{G(\lambda t)}{G(t)}=\lambda^p,\qquad\text{uniformly in $\lambda$ on each compact set of } [0,+\infty),
\end{equation}
and
\begin{equation}\label{premot}
\lim\limits_{\delta\to0^+}\frac{1}{\delta^{\sigma+1}\ell(\delta)}\int_0^\delta t^\sigma\ell(t)\,dt=\frac{1}{\sigma+1}.
\end{equation}
Thus, given $G\in\mathcal{RV}_p(0)$,
\begin{equation}\label{motivation}
\lim\limits_{\delta\to0^+} \frac{1}{G(\delta)}\int_0^\delta\frac{G(t)}{t}dt=\frac{1}{p}.
\end{equation}
As commented before, the prototype of Orlicz function fulfilling the required hypotheses is the function $G(t)=t^p(1+|\ln(t)|)$ with $p>1$. We have  $G(t)\in\mathcal{RV}_p(0)$, and it is obviously not comparable to any power $t^q$, $q>1$. By direct computation, $p^{\pm}=p\pm1$, and, moreover, $G'(t)$ belongs to $\mathcal{RV}_{p-1}(0)$. Thus, in what follows we will assume that the function $G(t)=t^p\ell(t)$ satisfies the following hypotheses
\begin{equation}\label{hypotheses}\tag{$\mathcal{H}$}
\left\{
\begin{tabular}{l}
$G\in \mathcal{RV}_{p}(0)$ is a piecewise smooth Orlicz function and $G'\in\mathcal{RV}_{p-1}(0)$,\\
$\ell(t)$ is bounded from below by a positive constant, say $\ell(t)\geq c_\ell>0$ for all $t>0$.
\end{tabular}
\right.
\end{equation}
\subsubsection{Fractional Orlicz-Sobolev spaces}\hfill

Given an Orlicz function $G$ the Orlicz space $L^G(\mathbb{R}^N)$ is defined as
\begin{equation*}
L^G(\mathbb{R}^N)\vcentcolon=\{u:\mathbb{R}^N\mapsto\mathbb{R},\ \text{measurable, such that } \Phi_G(u)<\infty\},
\end{equation*}
where
\begin{equation*}
\Phi_G(u)=\int_{\mathbb{R}^N} G(|u(x)|)dx.
\end{equation*}
Next, given $0<s\leq1$, the Orlicz-Sobolev space $W^{1,G}(\mathbb{R}^N)$ is defined as
\begin{equation*}
W^{s,G,\delta}(\mathbb{R}^N)\vcentcolon=\{u\in L^G(\mathbb{R}^N), \text{ such that } \Psi_{s,G,\delta}(u)<\infty\},
\end{equation*}
with
\begin{equation*}
\Psi_{s,G,\delta}(u)=\int_{\mathbb{R}^N}\int_{B(x,\delta)} G\left(\frac{|u(x)-u(y)|}{|x-y|^{s}}\right)\frac{dydx}{|x-y|^N},
\end{equation*}
for $0<s<1$, and 
\begin{equation*}
\Psi_{1,G,\delta}(u)=\Phi_{G}(|\nabla u|).
\end{equation*}
These spaces are endowed with the so-called Luxemburg norm,
\begin{equation*}
\|u\|_G=\|u\|_{L^G(\mathbb{R}^N)}\vcentcolon=\inf\limits_{\lambda>0}\left\{\Phi_G\left(\frac{u}{\lambda}\right)\leq1\right\}\qquad\text{and}\qquad\|u\|_{s,G,\delta}=\|u\|_{W^{s,G,\delta}(\mathbb{R}^N)}\vcentcolon=\|u\|_G+[u]_{s,G,\delta},
\end{equation*}
where
\begin{equation*}
[u]_{s,G,\delta}\vcentcolon=\inf\limits_{\lambda>0}\left\{\Phi_{s,G,\delta}\left(\frac{u}{\lambda}\right)\leq1\right\}.
\end{equation*}
Based on \cite{FernandezBonder2019}, given $\delta>0$ and $0<s<1$, the term $[u]_{s,G,\delta}$ is referred to as the $(s,G,\delta)$-Gagliardo seminorm. Note that, because of \eqref{hypotheses},
\begin{equation}\label{inclusion}
L^G(\mathbb{R}^N)\subset L^p(\mathbb{R}^N).
\end{equation}

 Next we recall some important properties of the spaces $L^G(\mathbb{R}^N)$ and $W^{1,G}(\mathbb{R}^N)$.  
\begin{theorem}
Let $G$ be an Orlicz function according to Definition \ref{def:orlicz}. Then the spaces $L^G(\mathbb{R}^N)$ and $W^{1,G}(\mathbb{R}^N)$ are reflexive, separable Banach spaces. Moreover, the dual space of $L^G(\mathbb{R}^N)$ can be identified with $L^{G^*}(\mathbb{R}^N)$. Finally, $C_c^{\infty}(\mathbb{R}^N)$ is dense in both $L^G(\mathbb{R}^N)$ and $W^{s,G}(\mathbb{R}^N)$.
\end{theorem}

Many of these properties of Orlicz-Sobolev spaces are obtained by very straightforward generalization of the proofs of the same properties for ordinary Sobolev spaces. In particular, the $\Delta_2$-condition ensures that the space $W^{1,G}(\mathbb{R}^N)$ is separable, while the $\nabla_2$-condition, \eqref{nabla_condition}, ensures that the space $W^{1,G}(\mathbb{R}^N)$ is reflexive. 
As indicated in \cite[Proposition 2.11]{FernandezBonder2019}, it is also straightforward to extend these functional properties to the space $W^{s,G,\infty}(\mathbb{R}^N)$. By the very definition of the spaces $W^{s,G,\infty}(\mathbb{R}^N)$, the same applies here. In particular, the density of $C_c^\infty(\mathbb{R}^N)$ follows by an standard argument of truncation and mollifier regularization jointly with convexity or Jensen's inequality, (cf. \cite[Proposition 2.11]{FernandezBonder2019}).
\begin{theorem}
Let $G$ be an Orlicz function according Definition \ref{def:orlicz} and $\delta>0$ and $0<s<1$. Then $W^{s,G,\delta}(\mathbb{R}^N)$ is a reflexive and separable Banach space. Moreover, the space $C_c^{\infty}(\mathbb{R}^N)$ is dense in $W^{s,G,\delta}(\mathbb{R}^N)$.
\end{theorem}
Following \cite[Theorem 3.1]{FernandezBonder2019} it can be also proved a Rellich--Kondrakov-type theorem for the spaces $W^{s,G,\delta}(\mathbb{R}^N)$, that is, the compactness of the embedding $W^{s,G,\delta}(\mathbb{R}^N)\hookrightarrow L^G(\mathbb{R}^N)$. 
\begin{theorem}\label{teo:compact}
Let $\delta>0$, $0<s<1$ and $G$ an Orlicz function. Then, for every $\{u_k\}_{k\in\mathbb{N}}\subset W^{s,G,\delta}(\mathbb{R}^N)$ bounded sequence, i.e., $\sup_k\|u_k\|_{s,G,\delta}<\infty$, there exists $u\in W^{s,G,\delta}(\mathbb{R}^N)$ such that, up to a subsequence, $u_k\to$
 in $L_{loc}^G(\mathbb{R}^N)$.
 \end{theorem}
 The proof of Theorem \ref{teo:compact} lies in proving the following equicontinuity estimate in order to apply a variant of the well-known Frèchet-Kolmogorov Compactness Theorem. Alhought the proof follows as in \cite[Lemma 3.2]{FernandezBonder2019} we include it for the sake of completeness.
 \begin{lemma}\label{lem:equicont}
 Let $\delta>0$, $0<s<1$ and $G$ be an Orlicz function. Then, there exists a constant $C>0$ such that 
 \begin{equation*}
 \Phi_G(u(x+h)-u(x))\leq C |h|^s\Psi_{s,G,\delta}(u),
 \end{equation*}
 for all $u\in W^{s,G,\delta}(\mathbb{R}^N)$ and every $0<|h|<\min\{\frac{1}{2},\delta\}$.
 \end{lemma}
 \begin{proof}
Because of $H_2$),
 \begin{equation*}
G(|u(x+h)-u(x)|)\leq \mathfrak{c}\Big( G(|u(x+h)-u(y)|) + G(|u(x)-u(y)|)\Big)
 \end{equation*}
 for all $y\in B(x,|h|)$. Then,
 \begin{equation*}
 \begin{split}
  \Phi_G(u(x+h)-u(x))&=\int_{\mathbb{R}^N} G(|u(x+h)-u(x)|) dx\\
  &=\frac{1}{|B(x,|h|)|}\int_{B(x,|h|)}\int_{\mathbb{R}^N} G(|u(x+h)-u(x)|) dx dy\\
  &\leq \frac{\mathfrak{c}}{\omega_N |h|^N}\left(\int_{B(x,|h|)}\int_{\mathbb{R}^N} G(|u(x+h)-u(y)|) dx dy\right.\\
  &\mkern+100mu \left.+\int_{B(x,|h|)}\int_{\mathbb{R}^N}G(|u(x)-u(y)|)dxdy\right)\\
 &=\frac{\mathfrak{c}}{\omega_N |h|^N}(I_1+I_2).
 \end{split}
 \end{equation*}
 Next, observe that, given $x\in\mathbb{R}^N$ and $y\in B(x,|h|)$, we have $ |x-y|\leq|h|$ and $|x+h-y|\leq|x-y|+|h|\leq2|h|$. Then, because of $P_5$) and the monotonicity of $G$, we get
 \begin{equation*}
 \begin{split}
 I_1&=\int_{B(x,|h|)}\int_{\mathbb{R}^N} G\left(\frac{|u(x+h)-u(y)|}{|x+h-y|^s}|x+h-y|^s\right) |x+h-y|^N\frac{dx}{|x+h-y|^N}dy\\
 &\leq \int_{B(x,|h|)}\int_{\mathbb{R}^N} G\left(\frac{|u(x+h)-u(y)|}{|x+h-y|^s}2^s|h|^s\right) 2^N|h|^N\frac{dx}{|x+h-y|^N}dy\\
 &\leq \int_{B(x,|h|)}\int_{\mathbb{R}^N} G\left(\frac{|u(x+h)-u(y)|}{|x+h-y|^s}\right) 2^{N+s}|h|^{N+s}\frac{dx}{|x+h-y|^N}dy\\
 &=2^{N+s}|h|^{N+s}\int_{B(x,|h|)}\int_{\mathbb{R}^N} G\left(\frac{|u(x)-u(y)|}{|x-y|^s}\right)\frac{dx}{|x-y|^N}dy\\
 &\leq 2^{N+s}|h|^{N+s}\int_{B(x,\delta)}\int_{\mathbb{R}^N} G\left(\frac{|u(x)-u(y)|}{|x-y|^s}\right)\frac{dx}{|x-y|^N}dy\\
 &= 2^{N+s}|h|^{N+s}\int_{B(0,\delta)}\int_{\mathbb{R}^N} G\left(\frac{|u(x+z)-u(x)|}{|z|^s}\right)\frac{dx}{|z|^N}dz\\
 &= 2^{N+s}|h|^{N+s}\int_{\mathbb{R}^N} \int_{B(0,\delta)} G\left(\frac{|u(x+z)-u(x)|}{|z|^s}\right)\frac{dzdx}{|z|^N}\\
 &=2^{N+s}|h|^{N+s}\Psi_{s,G,\delta}(u).
 \end{split}
 \end{equation*}
 Similarly, we can prove $ I_2\leq |h|^{N+s}\Psi_{s,G,\delta}(u)$ and the result follows. 
\end{proof}
\begin{proof}[Proof of Theorem \ref{teo:compact}]
Since $\{u_k\}_{k\in\mathbb{N}}$ is bounded in $W^{s,G,\delta}(\mathbb{R}^N)$ it is also bounded in $L^G(\mathbb{R}^N)$. Let us set 
$$M\vcentcolon=\sup_k\Big(\Psi_{s,G,\delta}(u_k)+\Phi_G(u_k)\Big).$$
Thus, by Lemma \ref{lem:equicont},
$$\sup_k\Phi_G\Big(u_k(x+h)-u(x)\Big)\leq C\, M\ |h|^s.$$
 By \cite[Theorem 11.5]{Krasn}, there exists $u\in L^G(\mathbb{R}^N)$ such that, up to a subsequence, $u_k\to u$ in $L_{loc}^G(\mathbb{R}^N)$. Moreover, $u\in W^{s,G,\delta}(\mathbb{R}^N)$ and, up to a subsequence, $u_k\to u$ a.e. in $\mathbb{R}^N$. Then,
 \begin{equation*}
 0\leq \lim\limits_{k\to+\infty}G\left(\frac{|u_k(x)-u_k(y)|}{|x-y|^s}\right)=G\left(\frac{|u(x)-u(y)|}{|x-y|^s}\right)\qquad\text{a.e. in }\mathbb{R}^N\times B(x,\delta).
 \end{equation*}
By Fatou's Lemma and the lower semicontinuity of $G$, we get
\begin{equation*}
\begin{split}
\Psi_{s,G,\delta}(u)
&\leq\liminf\limits_{k\to+\infty}\int_{\mathbb{R}^N}\int_{B(x,\delta)} G\left(\frac{|u_k(x)-u_k(y)|}{|x-y|^{s}}\right)\frac{dydx}{|x-y|^N}\leq\sup\limits_{k\in\mathbb{N}}\Psi_{s,G,\delta}(u_k)\leq M<+\infty.
\end{split}
\end{equation*}
As a consequence, $u\in W^{s,G,\delta}(\mathbb{R}^N)$.
 \end{proof}

Once we have introduced the main definitions we finish this section by exposing a first approach that motivates Theorem \ref{mainTheorem}. Actually, the identity \eqref{motivation} combined with Theorem \ref{ThBBM}
produces the following. Let us take a smooth function $u$ and consider
\begin{equation*}
\begin{split}
\Psi_{s,G,\delta}(u)&=\int_{\mathbb{R}^N}\int_{B(x,\delta)} G\left(\frac{|u(x)-u(y)|}{|x-y|^{s}}\right)\frac{dydx}{|x-y|^N}\\
&=\int_{\mathbb{R}^N}\int_{ B(x,\delta)}\frac{|u(x)-u(y)|^p}{|x-y|^{sp}}\ell\left(\frac{|u(x)-u(y)|}{|x-y|^{s}}\right)\frac{dydx}{|x-y|^N}.
\end{split}
\end{equation*}
Since
\begin{equation*}
\lim\limits_{\delta\to0^+}\sup\limits_{y\in B(x,\delta)}\Big||u(x)-u(y)|-|\nabla u(x)||x-y|\Big|=0,
\end{equation*}
it follows that
\begin{equation*}
\begin{split}
\lim\limits_{\delta\to0^+}\Psi_{s,G,\delta}(u)&=\lim\limits_{\delta\to0^+}\int_{\mathbb{R}^N}\int_{ B(x,\delta)} G\left(\frac{|u(x)-u(y)|}{|x-y|^{s}}\right)\frac{dydx}{|x-y|^N}\\
&=\lim\limits_{\delta\to0^+}\int_{\mathbb{R}^N}\int_{ B(x,\delta)}\frac{|u(x)-u(y)|^p}{|x-y|^{sp}}\ell\left(|\nabla u(x)||x-y|^{1-s}\right)\frac{dydx}{|x-y|^N}.
\end{split}
\end{equation*}
Following \cite{Bourgain2001}, let us define
\begin{equation*}
\rho_{\delta}(z,a)\vcentcolon=\frac{\ell(a|z|)}{|x|^{N-p(1-s)}}\chi_{B(0,\delta)}.
\end{equation*}
By direct computation,
\begin{equation*}
\int \rho_\delta(z,a)dz=\frac{1}{a^p(1-s)}\int_{0}^{a\delta^{1-s}}\frac{t^p\ell(t)}{t}dt,
\end{equation*}
so that, by \eqref{premot}, 
\begin{equation*}
\begin{split}
\lim\limits_{\delta\to0^+}\frac{1}{G(\delta^{1-s})}\int \rho_\delta(z,a)dz&=\frac{1}{a^p(1-s)}\lim\limits_{\delta\to0^+}\frac{1}{G(\delta^{1-s})}\int_{0}^{a\delta^{1-s}}\frac{t^p\ell(t)}{t}dt\\
&=\frac{1}{a^p(1-s)}\lim\limits_{\delta\to0^+} \frac{1}{G(\delta^{1-s})}\frac{(a\delta^{1-s})^p\ell(a\delta^{1-s})}{p}\\
&=\frac{1}{p(1-s)}\lim\limits_{\delta\to0^+} \frac{(\delta^{1-s})^p\ell(\delta^{1-s})}{G(\delta^{1-s})}\\
&=\frac{1}{p(1-s)},
\end{split}
\end{equation*}
where we have used the fact that $\ell(t)$ is a slowly varying function. Then, setting
\begin{equation*}
\overline{\rho}_\delta(z,a)=\frac{p(1-s)}{G(\delta^{1-s})}\rho_\delta(z,a),
\end{equation*}
we get 
\begin{equation*}
\lim\limits_{\delta\to0^+}\int\overline{\rho}_{\delta}(z,a)dz=1,\qquad\text{for all }a\geq0.
\end{equation*}
Hence, by Theorem \ref{ThBBM}, we conclude
\begin{equation*}
\begin{split}
\lim\limits_{\delta\to0^+}\frac{p(1-s)}{G(\delta^{1-s})}\Psi_{s,G,\delta}(u)&=\lim\limits_{\delta\to0^+}\int_{\mathbb{R}^N}\int_{B(x,\delta)} \frac{|u(x)-u(y)|^p}{|x-y|^{p}}\overline{\rho}_\delta(|x-y|,|\nabla u(x)|)dydx\\
&=K_{n,p}\int_{\mathbb{R}^N} |\nabla u(x)|^pdx.
\end{split}
\end{equation*}
\subsection{Technical Results}\label{Section:Technical}\hfill\newline
In this section we provide some estimates for the regularized families associated to a given function $u\in L^G(\Omega)$ useful in what follows.

Let $\rho\in C_c^{\infty}(\mathbb{R}^N)$ be an standard mollifier with $supp(\rho)=B(0,1)$ and let $\rho_r(x)=r^{-N}\rho\left(\frac{x}{r}\right)$ the approximation to the identity, so that the family $\{\rho_r\}_{r>0}$ is a family of positive functions such that 
\begin{equation}\label{mollifier_properties}
\rho_r\in C_c^{\infty}(\mathbb{R}),\qquad supp(\rho_r)=B(0,r)\qquad\text{and }\int_{\mathbb{R}^N}\rho_r\,dx=1.
\end{equation}
For each $u\in L^G(\Omega)$ and $r>0$ we define the mollified function $u_r\in L^G(\Omega)\cap C^{\infty}(\mathbb{R}^N)$ as
\begin{equation}\label{mollified}
u_r=u(x)\ast \rho_r(x).
\end{equation}
Note that, if $u\in L^G(\mathbb{R}^N)$, in particular, $u\in L_{loc}^1(\mathbb{R}^N)$, so that 
\begin{equation}\label{conv_ae}
u_r\to u\quad\text{in }L_{loc}^1(\mathbb{R}^N),\quad\text{and a.e. in }\mathbb{R}^N.
\end{equation}
If, moreover, $u\in W^{1,G}(\mathbb{R})$, in particular $u\in W^{1,1}_{loc}(\mathbb{R}^N)$, and then
\begin{equation}\label{conv_ae2}
\nabla u_k\to \nabla u\qquad\text{a.e. in }\mathbb{R}^N.
\end{equation}
The following useful estimate for the mollified function holds.
\begin{lemma}\label{lem:mollifier}
Let $u\in L^G(\Omega)$ and $\{u_r\}_{r>0}$ be the family defined as in \eqref{mollified}. Then
\begin{equation*}
\Psi_{s,G,\delta}(u_r)\leq\Psi_{s,G,\delta}(u).
\end{equation*}
\end{lemma}
\begin{proof}
First, let us note that
\begin{equation*}
\begin{split}
\Psi_{s,G,\delta}(u_r)&=\int_{\mathbb{R}^N}\int_{B(x,\delta)} G\left(\frac{|u_r(x)-u_r(y)|}{|x-y|^{s}}\right)\frac{dydx}{|x-y|^N}\\
&=\int_{\mathbb{R}^N}\int_{B(0,\delta)} G\left(\frac{|u_r(x+h)-u_r(x)|}{|h|^{s}}\right)\frac{dh}{|h|^N}dx\\
&=\int_{B(0,\delta)} \int_{\mathbb{R}^N}G\left(\frac{|u_r(x+h)-u_r(x)|}{|h|^{s}}\right)dx\frac{dh}{|h|^N}.
\end{split}
\end{equation*}
Moreover, since $G$ is convex, it follows that
\begin{equation*}
\begin{split}
G\left(\frac{|u_r(x+h)-u_r(x)|}{|h|^{s}}\right)&=G\left(\left|\int_{\mathbb{R}^N} u(x+h-z)-u(x-z)\rho_r(z)|h|^{-s}dz \right|\right)\\
&\leq\int_{\mathbb{R}^N}G\Big(|u(x+h-z)-u(x-z)||h|^{-s}\Big)\rho_r(z)dz.
\end{split}
\end{equation*}
Thus, because of \eqref{mollifier_properties},

\begin{equation*}
\begin{split}
\int_{\mathbb{R}^N}G\left(\frac{|u_r(x+h)-u_r(x)|}{|h|^{s}}\right)dx&\leq \int_{\mathbb{R}^N} \int_{\mathbb{R}^N}G\Big(|u(x+h-z)-u(x-z)||h|^{-s}\Big)\rho_r(z)dzdx\\
&=\int_{\mathbb{R}^N} \left(\int_{\mathbb{R}^N}G\Big(|u(x+h-z)-u(x-z)||h|^{-s}\Big)dx\right)\rho_r(z)dz\\
&=\int_{\mathbb{R}^N} \left(\int_{\mathbb{R}^N}G\Big(|u(x+h)-u(x)||h|^{-s}\Big)dx\right)\rho_r(z)dz\\
&=\int_{\mathbb{R}^N}G\Big(|u(x+h)-u(x)||h|^{-s}\Big)dx.
\end{split}
\end{equation*}
Then, we conclude
\begin{equation*}
\begin{split}
\Psi_{s,G,\delta}(u_r)&=\int_{B(0,\delta)} \int_{\mathbb{R}^N}G\left(\frac{|u_r(x+h)-u_r(x)|}{|h|^{s}}\right)dx\frac{dh}{|h|^N}\\
&\leq \int_{B(0,\delta)} \int_{\mathbb{R}^N}G\Big(|u(x+h)-u(x)||h|^{-s}\Big)dx \frac{dh}{|h|^N}\\
&=\int_{\mathbb{R}^N}\int_{B(0,\delta)} G\Big(|u(x+h)-u(x)||h|^{-s}\Big)\frac{dh}{|h|^N}dx\\
&=\int_{\mathbb{R}^N}\int_{B(x,\delta)} G\left(\frac{|u(x)-u(y)|}{|x-y|^{s}}\right)\frac{dydx}{|x-y|^N}\\
&=\Psi_{s,G,\delta}(u).
\end{split}
\end{equation*}
\end{proof}
Next, we provide estimates for truncated functions. Let $\eta\in C^{\infty}(\mathbb{R}^N)$ such that $\eta\equiv 1$ in $B(0,1)$, $supp(\eta)=B(0,2)$ with $0\leq\eta\leq 1$ in $\mathbb{R}^N$ and $\|\nabla \eta\|_{\infty}\leq 2$. Given $k\in\mathbb{N}$, let us define $\eta_k(x)=\eta\left(\frac{x}{k}\right)$. Note that $\{\eta_k\}_{k\in\mathbb{N}}\in C_c^{\infty}(\mathbb{R}^N)$ and, for $k\in\mathbb{N}$, we have
\begin{equation}\label{truncation_properties}
0\leq\eta_k\leq1,\quad \eta_k=1\in B(0,k),\quad supp(\eta_k)=B(0,2k),\quad\text{and}\quad|\nabla\eta_k|\leq\frac{2}{k}.
\end{equation}
Given $u\in L^G(\Omega)$ we define the truncated functions $u_k$, $k\in\mathbb{N}$ as
\begin{equation}\label{truncated}
u_k=\eta_k u.
\end{equation}
We prove next an estimate on the truncated functions.
\begin{lemma}
Let $u\in L^G(\mathbb{R}^N)$ and $\{u_k\}_{k\in\mathbb{N}}$ defined as in \eqref{truncated} above. Then
\begin{equation*}
\Psi_{s,G,\delta}(u_k)\leq\frac{\mathfrak{c}}{2}\Psi_{s,G,\delta}(u)+\frac{N\omega_N\mathfrak{c}^2}{2k(1-s)}\Phi_G(u\,\delta^{1-s}).
\end{equation*}
\end{lemma}
\begin{proof}
Since $0\leq\eta_k\leq1$, from $P_4$), we have (cf. \cite[Lemma 2.14]{FernandezBonder2019}),
\begin{equation*}
G\left(\frac{|u_k(x)-u_k(y)|}{|x-y|^{s}}\right)\leq\frac{\mathfrak{c}}{2}G\left(\frac{|u(x)-u(y)|}{|x-y|^{s}}\right)+\frac{\mathfrak{c}}{2}G\left(\frac{|u(x)||\eta_k(x)-\eta_k(y)|}{|x-y|^{s}}\right).
\end{equation*}
Thus, by  $P_5$), $H_2$) and \eqref{truncation_properties}, we conclude

\begin{equation*}
\begin{split}
\int_{\mathbb{R}^N}\int_{B(x,\delta)}G\left(\frac{|u(x)||\eta_k(x)-\eta_k(y)|}{|x-y|^{s}}\right)\frac{dydx}{|x-y|^N}&\leq \int_{\mathbb{R}^N}\int_{B(x,\delta)}G\left(\frac{2}{k}|u(x)||x-y|^{1-s}\right)\frac{dydx}{|x-y|^N}\\
&\leq \frac{\mathfrak{c}}{k}\int_{\mathbb{R}^N}\int_{ B(x,\delta)}G\left(|u(x)||x-y|^{1-s}\right)\frac{dydx}{|x-y|^N}\\
&\leq \frac{\mathfrak{c}}{k}\int_{\mathbb{R}^N}\int_{B(x,1)}G\left(|u(x)|\delta^{1-s}|x-y|^{1-s}\right)\frac{dydx}{|x-y|^N}\\
&\leq \frac{\mathfrak{c}}{k}\int_{\mathbb{R}^N}\int_{B(x,1)}G\left(|u(x)|\delta^{1-s}\right)|x-y|^{1-s}\frac{dydx}{|x-y|^N}\\
&\leq \frac{\mathfrak{c}}{k}\int_{\mathbb{R}^N}\int_{B(0,1)}G\left(|u(x)|\delta^{1-s}\right)|w|^{1-s}\frac{dwdx}{|w|^N}\\
&=\frac{\mathfrak{c}N\omega_N}{k(1-s)}\int_{\mathbb{R}^N}G\left(|u(x)|\delta^{1-s}\right)dx,
\end{split}
\end{equation*}
where $\omega_N$ denotes the surface of the unitary sphere $\mathbb{S}^{N-1}$. 
\end{proof}
\section{Asymptotic Behavior}\label{Section:Asym}
We begin with the following result that contains the key points of the limit process $\delta\to0^+$. Its proof, strongly relying on Theorem \ref{teo:karamata}, clarifies the requirement $G\in \mathcal{RV}_p(0)$ as well as the reason to consider scaling $\frac{p(1-s)}{G(\delta^{1-s})}$.
\begin{proposition}\label{mainprop1}
Let $G$ be an Orlicz function satisfying \eqref{hypotheses} and $u\in C_c^2(\mathbb{R}^N)$. Then, for every $x\in\mathbb{R}^N$, we have
\begin{equation*}
\lim\limits_{\delta\to0^+}\frac{p(1-s)}{G(\delta^{1-s})}\int_{B(x,\delta)}G\left(\frac{|u(x)-u(y)|}{|x-y|^{s}}\right)\frac{dy}{|x-y|^N}=K_{N,p}|\nabla u(x)|^{p}.
\end{equation*}
\end{proposition}

\begin{proof}
For $x\in\mathbb{R}^N$ and $y\in B(x,\delta)$, $x\neq y$, we have
\begin{equation*}
\begin{split}
\left|G\left(\frac{|u(x)-u(y)|}{|x-y|^{s}}\right)-G\left(\left|\nabla u(x)\cdot\frac{x-y}{|x-y|^{s}}\right|\right)\right|
&\leq G'(\xi_M)\frac{|u(x)-u(y)-\nabla u(x)\cdot(x-y)|}{|x-y|^s}\\
&\leq c G'(\xi_M) |x-y|^{2-s},
\end{split}
\end{equation*}
for some $c>0$ depending on the $C^2$-norm of $u$ and some $\xi_M\in[0,\|\nabla u\|_{\infty}\delta^{1-s}]$ such that
\begin{equation*}
G'(\xi_M)=\max\limits_{\xi\in[0,\|\nabla u\|_{\infty}\delta^{1-s}]}G'(\xi).
\end{equation*}
Then,
\begin{equation*}
\begin{split}
\frac{1}{G(\delta^{1-s})}\!\int_{B(x,\delta)}\! \left|G\left(\frac{|u(x)-u(y)|}{|x-y|^{s}}\right)\!-G\!\left(\left|\nabla u(x)\cdot\frac{x-y}{|x-y|^{s}}\right|\right)\!\right|\frac{dy}{|x-y|^N}&\leq c\, \frac{G'(\xi_M) \delta^{2-s}}{G(\delta^{1-s})}\\
&=c\,\frac{G'(\xi_M)}{G'(\delta^{1-s})}\frac{\delta^{1-s}G'(\delta^{1-s})}{G(\delta^{1-s})}\delta\\
&\leq c\,p^+\frac{G'(\xi_M)}{G'(\delta^{1-s})}\delta,
\end{split}
\end{equation*}
for $p^+$ given by \eqref{pbounds}. Taking $\lambda_M\in[0,\|\nabla u\|_{\infty}]$ such that $\xi_M=\lambda_M\delta^{1-s}$, by \eqref{hypotheses}, we get
\begin{equation*}
\lim\limits_{\delta\to0^+} \frac{G'(\xi_M)}{G'(\delta^{1-s})}=\lim\limits_{\delta\to0^+} \frac{G'(\lambda_M\delta^{1-s})}{G'(\delta^{1-s})}=\lambda_M^{p-1}\leq\|\nabla u\|_{\infty}^{p-1}.
\end{equation*}
Then, it follows that
\begin{equation*}
\lim\limits_{\delta\to0^+}\frac{1}{G(\delta^{1-s})}\Psi_{s,G,\delta}(u)=\lim\limits_{\delta\to0^+}\frac{1}{G(\delta^{1-s})}\int_{ B(x,\delta)}G\left(\left|\nabla u(x)\cdot\frac{x-y}{|x-y|^{s}}\right|\right)\frac{dy}{|x-y|^{N}}.
\end{equation*}
Next, observe that, by the symmetry of the integral on $B(x,\delta)$, we have
\begin{equation}\label{eq:mainstep}
\begin{split}
\int_{B(x,\delta)}G\left(\left|\nabla u(x)\cdot\frac{x-y}{|x-y|^{s}}\right|\right)\frac{dy}{|x-y|^{N}}&=\int_{B(x,\delta)}G\left(|\nabla u(x)|\left|e_x\cdot\frac{x-y}{|x-y|^{s}}\right|\right)\frac{dy}{|x-y|^{N}}\\
&=\int_{B(0,\delta)}G\left(|\nabla u(x)|\left|e\cdot\frac{h}{|h|^{s}}\right|\right)\frac{dh}{|h|^{N}}\\
&=\int_{B(0,1)}G\left(|\nabla u(x)|\ \delta^{1-s}\ \left|e\cdot\frac{w}{|w|^{s}}\right|\right)\frac{dw}{|w|^{N}},
\end{split}
\end{equation}
where $e_x,\ e_0$ are unitary vectors with origin at $x\in\mathbb{R}^N$ and $0\in\mathbb{R}^N$ respectively. Therefore, since $G\in\mathcal{RV}_{p}(0)$, because of \eqref{unifconvergence}, we conclude
\begin{equation*}
\begin{split}
\lim\limits_{\delta\to0^+}\!\frac{1}{G(\delta^{1-s})}\!\int_{B(0,1)}\!G\!\left(|\nabla u(x)|\ \delta^{1-s}\left|e\cdot\frac{w}{|w|^{s}}\right|\right)\!\frac{dw}{|w|^{N}}\!&=\!\int_{B(0,1)}\lim\limits_{\delta\to0^+}\!\!\frac{G\left(|\nabla u(x)|\ \delta^{1-s}\ \left|e\cdot\frac{w}{|w|^{s}}\right|\right)}{G(\delta^{1-s})}\frac{dw}{|w|^{N}}\\
&=\int_{B(0,1)}|\nabla u(x)|^p\left|e\cdot\frac{w}{|w|^{s}}\right|^p\frac{dw}{|w|^N}\\
&=\frac{K_{N,p}}{p(1-s)}|\nabla u(x)|^p.
\end{split}
\end{equation*}
\end{proof}
\begin{remark} Following \cite[Proposition 9]{Bellido2015}, the above Proposition \ref{mainprop1} can be relaxed by considering functions with bounded gradient. Say, let $\mathcal{F}\subset C^1(\mathbb{R}^N)$ be such that $\{\nabla u:\, u\in\mathcal{F}\}$ is bounded in $C^1(\mathbb{R}^N)$ and equicontinuous. Then,
\begin{equation*}
\lim\limits_{\delta\to0^+}\sup\limits_{u\in\mathcal{F}}\left|\frac{1}{G(\delta^{1-s})}\int_{ B(x,\delta)}G\left(\frac{|u(x)-u(y)|}{|x-y|^{s}}\right)\frac{dy}{|x-y|^N}-K_{N,p}|\nabla u|^{p}\right|=0.
\end{equation*}
Let $M$ be the supremum of the Lipschitz constants of $u\in\mathcal{F}$. Since the set $\{\nabla u:\, u\in\mathcal{F}\}$ is equicontinuous, there exists $\mu:[0,\infty)\mapsto[0,\infty)$ such that $\mu(t)\to\mu(0)=0$ as $t\to0^+$, and $|\nabla u(x)-\nabla u(y)|\leq\mu(|x-y|)$, for all $u\in\mathcal{F}$. Hence, if $y\in B(x,\delta)$, we have
\begin{equation*}
\begin{split}
\left|\frac{u(x)-u(y)-\nabla u(x)\cdot(x-y)}{|x-y|^s}\right|&\leq\int_0^1|\nabla u(x-t(x-y))-\nabla u(x)|dt\ |x-y|^{1-s}\\
&\leq \int_0^1\mu(t|x-y|)dt\ |x-y|^{1-s}\\
&\leq\mu(\delta)\delta^{1-s}.
\end{split}
\end{equation*}
Arguing as in Proposition \ref{mainprop1}, we conclude, 
\begin{equation*}
\int_{B(x,\delta)} \left|G\left(\frac{|u(x)-u(y)|}{|x-y|^{s}}\right)-G\left(\left|\nabla u(x)\cdot\frac{x-y}{|x-y|^{s}}\right|\right)\right|\frac{dy}{|x-y|^{N}}\leq c\,G'(\xi_M)\mu(\delta)\delta^{1-s},
\end{equation*}
for some constant $c>0$ and some $\xi_M\in[0,\|\nabla u\|_{\infty}\delta^{1-s}]$. Moreover, by \eqref{pbounds},
\begin{equation*}
\lim\limits_{\delta\to0^+}\frac{G'(\xi_M)\mu(\delta)\delta^{1-s}}{G(\delta^{1-s})}\leq p^+ \lim\limits_{\delta\to0^+} \frac{G'(\xi_M)}{G'(\delta^{1-s})}\mu(\delta)\leq p^+ \|\nabla u\|_{\infty}^p \lim\limits_{\delta\to0^+} \mu(\delta)=0,
\end{equation*} 
and Proposition \ref{mainprop1} follows under assumption $u\in C^1(\mathbb{R}^N)$.
\end{remark}
If one deals with truncated functionals of the form \eqref{correa}, namely
\begin{equation*}
\int_{\mathbb{R}^N}\int_{B(x,\delta)} G(|u(x)-u(y)|)J(|x-y|)dydx,
\end{equation*}
the proof of Proposition \ref{mainprop1} can be easily adapted to prove the following.
\begin{proposition}\label{prop:correa}
Let $G$ be an Orlicz function satisfying \eqref{hypotheses}, a kernel $J\in\mathcal{RV}_q(0)$ for some $q\in\mathbb{R}$ with $0<N+p+q$ and $u\in C_c^2(\mathbb{R}^N)$. Then, for every $x\in\mathbb{R}^N$,
\begin{equation*}
\lim\limits_{\delta\to0^+}\frac{N+p+q}{G(\delta)J(\delta)\delta^N}\int_{\mathbb{R}^N}\int_{B(x,\delta)} G(|u(x)-u(y)|)J(|x-y|)dydx=K_{N,p}\int_{\mathbb{R}^N}|\nabla u(x)|^p dx.
\end{equation*}
\end{proposition}
\begin{proof}
Arguing as in Proposition \ref{mainprop1}, the analogue of \eqref{eq:mainstep} reads
\begin{equation*}
\begin{split}
\int_{B(x,\delta)}G\Big(|\nabla u(x)\cdot(y-x)|\Big)J(|x-y|)dy&=\int_{B(x,\delta)}G\Big(|\nabla u(x)|\, |e_x\cdot(y-x)|\Big)J(|x-y|)dy\\
&=\int_{B(0,\delta)}G\left(|\nabla u(x)|\left|e\cdot h\right|\right) J(|h|)dh\\
&=\int_{B(0,1)}G\Big(|\nabla u(x)|\ \delta\ |e\cdot w|\Big)J(\delta|w|)\delta^Ndh,
\end{split}
\end{equation*}
so that, if $G$ satisfies \eqref{hypotheses} and we assume $J\in\mathcal{RV}_q(0)$ for some $q\in\mathbb{R}$ with $0<N+p+q$, we get  
\begin{equation*}
\begin{split}
\lim\limits_{\delta\to0^+}\frac{1}{G(\delta)J(\delta)\delta^N}\int_{B(0,1)}&G\Big(|\nabla u(x)|\ \delta |e\cdot w|\Big)J(\delta|w|)\delta^Ndw,\\
&=\int_{B(0,1)}\lim\limits_{\delta\to0^+}\frac{G\Big(|\nabla u(x)|\ \delta  |e\cdot w|\Big)J(\delta|w|)}{G(\delta)J(\delta)}dh,\\
&=\int_{B(0,1)}|\nabla u(x)|^p  |e\cdot w|^p|w|^qdw,\\
&=\frac{K_{N,p}}{N+p+q}|\nabla u(x)|^p.
\end{split}
\end{equation*}
\end{proof}
The main object involved in the limit $\delta\to0^+$ in Proposition \ref{mainprop1} is given in \eqref{eq:mainstep}. Thus, given $a\in\mathbb{R}_+$, and an unitary vector $e$ we define
\begin{equation*}
\varphi(a,s,\delta)\vcentcolon=\int_{B(0,\delta)}G\left(a\left|e\cdot\frac{h}{|h|^{s}}\right|\right)\frac{dh}{|h|^N}.
\end{equation*}
Note that, by symmetry, $\varphi(a,s,\delta)$ is independent of the vector $e$. Moreover,
\begin{equation*}
\varphi(a,s,\delta)=\int_{B(x,1)}G\left(a\,\delta^{1-s}\left|e\cdot\frac{w}{|w|^{s}}\right|\right)\frac{dw}{|w|^N}=\int_0^1\int_{\mathbb{S}^{N-1}}G\left(a\,\delta^{1-s}r^{1-s}|w_N|\right)dS_w\frac{dr}{r}.
\end{equation*}
The dependence in $s$ and $\delta$ is then encoded in the function $\varphi(a,s,\delta)$ and hence, this function also appears in the characterization of the limit $s\to1^-$. Actually, by Theorem \ref{lim_s},
\begin{equation*}
\tilde{G}(a)\vcentcolon=\lim\limits_{s\to1^-}(1-s)\int_0^1\int_{\mathbb{S}^{N-1}}G\left(a\,r^{1-s}|w_N|\right)dS_w\frac{dr}{r}=\lim\limits_{s\to1^-}(1-s)\varphi(a,s,1).
\end{equation*}
It is clear then, that the appropriate scaling in $\delta$ is given by the behavior of $G(\delta^{1-s})$ as $\delta\to0^+$. Thus, in the localization of functionals as \eqref{genfunct} for $\delta\to0^+$, we are led to the study of  
\begin{equation*}
\lim\limits_{\delta\to0^+}\frac{G(\lambda \delta^{1-s})}{G(\delta^{1-s})},
\end{equation*}
with $\lambda=|\nabla u(x)| \left|v\cdot\frac{w}{|w|^{s}}\right|$. By Theorem \ref{teo:karamata}, the above limit exists if and only if $G\in\mathcal{RV}_p(0)$, for some $p\in\mathbb{R}$. So the localization process will be well-defined only for functionals as \eqref{genfunct} involving Orlicz functions $G$ such that $G\in\mathcal{RV}_p(0)$.

\begin{remark}\label{remark} Let us compute the function $\tilde{G}(a)$ in Theorem \ref{lim_s} for some choices of $G(t)$ to illustrate that, in general, the limits $s\to1^-$ and $\delta\to0^+$ produce different results.
\begin{itemize}
\item If $G(t)=t^p$, we have,
\begin{equation*}
\int_{\mathbb{S}^{N-1}}G\left(a\,r^{1-s}|w_N|\right)dS_w=a^pr^{p(1-s)}\int_{\mathbb{S}^{N-1}}|w_N|^pdS_w=K_{N,p}\,a^pr^{p(1-s)}.
\end{equation*}
Then,
\begin{equation*}
\tilde{G}(a)=\lim\limits_{s\to1^-}(1-s)\varphi(a,1)=\frac{K_{N,p}}{p}a^p
\end{equation*}
and 
\begin{equation*}
\lim\limits_{\delta\to0^+}\frac{p(1-s)}{G(\delta^{1-s})} \varphi(a,\delta) =K_{N,p}\,a^p.
\end{equation*}
Hence, both the limit $s\to1^-$ and the localization as $\delta\to0^+$ produce the same result up to a multiplicative constant.
\vspace{0.1cm}
\item If $G(t)=t^p(1+|\ln(t)|)=t^p+G_2(t)$, with $p>1$, we have (cf. \cite[Example 2.17-(2)]{FernandezBonder2019}),
\begin{equation*}
\int_{\mathbb{S}^{N-1}}G_2\left(a\,r^{1-s}|w_N|\right)dS_w=a^pr^{p(1-s)}\Big(K_{N,p}|\ln(a)|+K_{N,p,\ln}-(1-s)K_{N,p}\ln(r)\Big),
\end{equation*}
where 
\begin{equation*}
K_{N,p,\ln}=\int_{\mathbb{S}^{N-1}}|w_n|^p\,|\ln|w_N||dS_w.
\end{equation*}
Then, by direct computation, we get
\begin{equation}\label{eq:ex1}
\varphi(a,s,\delta)=\frac{a^p}{p(1-s)}\delta^{p(1-s)}\left(K_{N,p}+K_{N,p}|\ln(a)|+K_{N,p,\ln}+\frac{K_{N,p}}{p}\Big(1-p(1-s)\ln(\delta)\Big)\right),
\end{equation}
and, hence,
\begin{equation*}
\tilde{G}(a)=\lim\limits_{s\to1^-}(1-s)\varphi(a,1)=\frac{a^p}{p}\left(K_{N,p}+K_{N,p}|\ln(a)|+K_{N,p,\ln}+\frac{K_{N,p}}{p}\right).
\end{equation*}
Next we find the limit as $\delta\to0^+$. For $\delta<1$, the function $\varphi(a,s,\delta)$ given in \eqref{eq:ex1}, can be written as 
\begin{equation*}
\varphi(a,s,\delta)=\frac{a^p}{p(1-s)}\delta^{p(1-s)}\left(K_{N,p}+K_{N,p}|\ln(a)|+K_{N,p,\ln}+\frac{K_{N,p}}{p}\Big(1+p|\ln(\delta^{1-s})|\Big)\right).
\end{equation*}
The behavior of $\varphi(a,s,\delta)$ is clearly encoded by the term $\delta^{p(1-s)}\ln(\delta^{1-s})=G(\delta^{1-s})$, and
\begin{equation*}
\lim\limits_{\delta\to0^+}\frac{p(1-s)}{G(\delta^{1-s})}\varphi(a,\delta)=K_{N,p}\,a^p.
\end{equation*}
Thus, for $G(t)=t^p(1+|\ln(t)|)$ the localization under the limit $\delta\to0^+$ produces a different operator that the one obtained under the limit $s\to1^-$. 
\vspace{0.1cm}
\item If $G(t)=\max\{t^q,t^p\}$, with $1<q<p$, we have (cf. \cite[Example 2.17-(4))]{FernandezBonder2019})
\begin{equation*}
\tilde{G}(a)=
\left\{
\begin{tabular}{lr}
$\dfrac{K_{N,q}}{q}a^q$ & if $a\leq1$,\\[10pt]
$\displaystyle\frac{a^q}{q} \int_{|w_N|\leq\frac{1}{a}}|w_N|^qdS_w+\frac{a^p}{p}\int_{|w_N|>\frac{1}{a}}|w_N|^pdS_w+\left(\frac{1}{q}-\frac{1}{p}\right)\int_{|w_N|>\frac{1}{a}}dS_w$& if $a>1$,
\end{tabular}
\right.
\end{equation*}
while, by direct computation,
\begin{equation*}
\varphi(a,s,\delta)=
\left\{
\begin{tabular}{lr}
$\dfrac{K_{N,q}}{q(1-s)}a^q\delta^{q(1-s)}$ & if $a\leq \frac{1}{\delta^{1-s}}$,\\[10pt]
$\displaystyle\frac{a^q\delta^{q(1-s)}}{q(1-s)}\int_{|w_N|\leq\frac{1}{a\delta^{1-s}}}|w_N|^qdS_w+\frac{a^p\delta^{p(1-s)}}{p(1-s)}\int_{|w_N|>\frac{1}{a\delta^{1-s}}}|w_N|^pdS_w$&\\[10pt]
$\displaystyle+\left(\frac{1}{q(1-s)}\delta^{q(1-s)}-\frac{1}{p(1-s)}\delta^{p(1-s)}\right)\int_{|w_N|>\frac{1}{a\delta^{1-s}}}dS_w$& if $a>\frac{1}{\delta^{1-s}}$.
\end{tabular}
\right.
\end{equation*}
Since, for $\delta<1$, we have $G(\delta^{1-s})=\delta^{q(1-s)}$, then  
\begin{equation*}
\lim\limits_{\delta\to0^+}\frac{q(1-s)}{G(\delta^{1-s})}\varphi(a,s,\delta)=K_{N,q}\,a^q.
\end{equation*}
\end{itemize}
\end{remark}
The next Proposition \ref{mainprop2} corresponds to \cite[Lemma 3.4]{Alberico2020} and \cite[Lemma 4.2]{FernandezBonder2019}. Once again, note that the contribution of the tails of the function is neglected by the very definition of the problem. On the other hand, the role of $\delta$ appears as an scaling acting on the function $G$ according to the regularly varying functions scheme.
\begin{proposition}\label{mainprop2}
Let $u\in W^{1,G}(\mathbb{R}^N)$. Then, for $\delta>0$, we have
\begin{equation*}
\Psi_{s,G,\delta}(u)\leq\frac{N\omega_N}{1-s}\Phi_G(|\nabla u|\delta^{1-s}).
\end{equation*}
\end{proposition}

\begin{proof}
Since $u\in W^{1,G}(\mathbb{R}^N)$, in particular $u\in W^{1,1}_{loc}(\mathbb{R}^N)$, so that
\begin{equation*}
u(x+h)-u(x)=\int_0^1\nabla u(x+th)\cdot h\, dt\qquad\text{for a.e. } x,h\in\mathbb{R}^N
\end{equation*}
and, thus
\begin{equation*}
\frac{|u(x+h)-u(x)|}{{|h|^s}}\leq\int_0^1|\nabla u(x+th)||h|^{1-s}dt.
\end{equation*}
Hence, as $G$ is convex, we get
\begin{equation*}
G\left(\frac{|u(x+h)-u(x)|}{|h|^{s}}\right)\leq G\left(\int_0^1|\nabla u(x+th)||h|^{1-s}dt\right)\leq \int_0^1G(|\nabla u(x+th)||h|^{1-s})dt.
\end{equation*}
Combining this with $P_5$), we conclude

\begin{equation*}
\begin{split}
\int_{\mathbb{R}^N}\int_{ B(0,\delta)}G\left(\frac{|u(x)-u(y)|}{|x-y|^{s}}\right)\frac{dy}{|x-y|^N}dx&=\int_{\mathbb{R}^N}\int_{ B(0,\delta)}G\left(\frac{|u(x+h)-u(x)|}{|h|^{s}}\right)\frac{dh}{|h|^N}dx\\
&\leq \int_{\mathbb{R}^N}\int_{ B(0,\delta)} \int_0^1G(|\nabla u(x+th)||h|^{1-s})dt\frac{dh}{|h|^N}dx\\
&=\int_0^1 \int_{\mathbb{R}^N}  \int_{ B(0,\delta)} G(|\nabla u(x+th)||h|^{1-s})\frac{dh}{|h|^N} dxdt\\
&=\int_0^1 \int_{\mathbb{R}^N}  \int_{ B(0,1)} G(|\nabla u(x+t\delta w)|\delta^{1-s}|w|^{1-s})\frac{dw}{|w|^N} dxdt\\
&\leq \int_0^1 \int_{\mathbb{R}^N}  \int_{ B(0,1)} G(|\nabla u(x+t\delta w)|\delta^{1-s})|w|^{1-s}\frac{dw}{|w|^N} dxdt\\
&= \int_0^1 \int_{ B(0,1)}  \int_{\mathbb{R}^N}  G(|\nabla u(x+t\delta w)|\delta^{1-s}) dx |w|^{1-s-N}dwdt\\
&=\int_0^1 \int_{ B(0,1)}  \int_{\mathbb{R}^N}  G(|\nabla u(x)|\delta^{1-s}) dx |w|^{1-s-N}dwdt\\
&= \int_{ B(0,1)} |w|^{1-s-N}dw \int_{\mathbb{R}^N}  G(|\nabla u(x)|\delta^{1-s}) dx \\
&=\frac{N\omega_N}{1-s}\Phi_G(|\nabla u|\delta^{1-s}).
\end{split}
\end{equation*}
\end{proof}
Next we address the $\liminf$ inequality.
\begin{proposition}\label{prop_liminf}
Let $u\in C^1(\mathbb{R}^N)$. Then, for every $x\in\mathbb{R}^N$, we have
\begin{equation*}
\liminf\limits_{\delta\to0^+}\frac{1}{G(\delta^{1-s})}\int_{B(x,\delta)}G\left(\frac{|u(x)-u(y)|}{|x-y|^{s}}\right)\frac{dy}{|x-y|^N}\geq K_{N,p}|\nabla u(x)|^{p}.
\end{equation*}
\end{proposition}
\begin{proof}
Given $\epsilon>0$, there exists $\delta>0$ such that, for all $x\in\mathbb{R}^N$ and $y\in B(x,\delta)$,
\begin{equation*}
|u(x)-u(y)-\nabla u(x)\cdot(y-x)|\leq\epsilon |x-y|,
\end{equation*}
so that
\begin{equation*}
\left|\nabla u(x)\cdot\frac{x-y}{|x-y|}\right||x-y|^{1-s}\leq\frac{|u(x)-u(y)|}{|x-y|^s}+\epsilon |x-y|^{1-s}.
\end{equation*}
Conversely, given $\delta>0$, let us set $\varepsilon=\varepsilon(\delta)>0$ be defined as
\begin{equation*}
\varepsilon=\min\{\epsilon>0:\ |u(x)-u(y)-\nabla u(x)\cdot(y-x)|\leq \epsilon |x-y|,\ \text{for }y\in B(x,\delta)\}.
\end{equation*}
Plainly, $\varepsilon$ is well-defined for functions $u\in C^1(\mathbb{R}^N)$ and, moreover, $\varepsilon$ is decreasing in $\delta$ and $\varepsilon\to0$ as $\delta\to0^+$. Next, fix $\eta\in(0,1)$. By using the convexity of $G$, we find

\begin{equation*}
\begin{split}
\int_{B(x,\delta)}G&\left(\left|\nabla u(x)\cdot\frac{x-y}{|x-y|}\right||x-y|^{1-s}\right)\frac{dy}{|x-y|^N}\\
&\leq \int_{B(x,\delta)}G\left(\frac{|u(x)-u(y)|}{(1-\eta)|x-y|^s}\right)\frac{dy}{|x-y|^N} +\int_{B(x,\delta)}G\left(\frac{\varepsilon}{\eta}|x-y|^{1-s}\right)\frac{dy}{|x-y|^N}\\
&= \int_{B(x,\delta)}G\left(\frac{|u(x)-u(y)|}{(1-\eta)|x-y|^s}\right)\frac{dy}{|x-y|^N} +\omega_N \int_0^\delta G\left(\frac{\varepsilon}{\eta}r^{1-s}\right)\frac{dr}{r}\\
&= \int_{B(x,\delta)}G\left(\frac{|u(x)-u(y)|}{(1-\eta)|x-y|^s}\right)\frac{dy}{|x-y|^N} +\frac{\omega_N}{1-s} \int_{0}^{\frac{\varepsilon}{\eta}\delta^{1-s}} \frac{G(t)}{t}dt.
\end{split}
\end{equation*} 
To continue, let us fix $\delta_0>0$ and take $\varepsilon_0=\varepsilon(\delta_0)$. Since, by Theorem \ref{KaramataDirect},
\begin{equation*}
\lim\limits_{\delta\to0^+}\frac{1}{G(\frac{\varepsilon_0}{\eta}\delta^{1-s} )}\int_{0}^{\frac{\varepsilon_0}{\eta}\delta^{1-s}} \frac{G(t)}{t}dt= \frac{1}{p},
\end{equation*}
it follows that
\begin{equation*}
\lim\limits_{\delta\to0^+}\frac{p}{G(\delta^{1-s})}\int_{0}^{\frac{\varepsilon_0}{\eta}\delta^{1-s}} \frac{G(t)}{t}= \left(\frac{\varepsilon_0}{\eta}\right)^p.
\end{equation*}
As a consequence, since $\varepsilon$ is decreasing in $\delta$,
\begin{equation*}
\lim\limits_{\delta\to0^+}\frac{p(1-s)}{G(\delta^{1-s})}\int_{0}^{\frac{\varepsilon}{\eta}\delta^{1-s}} \frac{G(t)}{t}\leq \left(\frac{\varepsilon_0}{\eta}\right)^p,\qquad \text{for all }\delta_0>0.
\end{equation*}
As $\delta_0>0$ is arbitrary and $\varepsilon\to 0$ as $\delta\to0^+$, we can take $\varepsilon_0$ arbitrarily small and, thus,
 \begin{equation*}
\lim\limits_{\delta\to0^+}\frac{p(1-s)}{G(\delta^{1-s})}\int_{0}^{\frac{\varepsilon}{\eta}\delta^{1-s}} \frac{G(t)}{t}=0.
\end{equation*}
Therefore,
\begin{equation*}
\begin{split}
|\nabla u(x)|^p&= \lim\limits_{\delta\to 0^+}\frac{1}{G(\delta^{1-s})} \int_{B(x,\delta)}G\left(\left|\nabla u(x)\cdot\frac{x-y}{|x-y|}\right||x-y|^{1-s}\right)\frac{dy}{|x-y|^N}\\
&\leq \liminf\limits_{\delta\to0^+} \int_{B(x,\delta)}G\left(\frac{|u(x)-u(y)|}{(1-\eta)|x-y|^s}\right)\frac{dy}{|x-y|^N}.
\end{split}
\end{equation*}
We finish by setting $v=\frac{1}{1-\eta}u$ and taking the limit $\eta\to0^+$ (cf. \cite[Lemma 3.3]{Alberico2020}).
\end{proof}
We are ready to prove the main result of this work.
\begin{proof}[Proof of Theorem \ref{mainTheorem}]
The proof follows as in \cite[Theorem 1.1]{Alberico2020}. Let $u\in W^{1,G}(\mathbb{R}^N)$ and let $\{u_k\}_{k\in\mathbb{N}}$ be the sequence defined by \eqref{mollified}. Note that, by \eqref{inclusion}, $u\in W^{1,p}(\mathbb{R}^N)$. Then, by Proposition \ref{mainprop2}, Lemma \ref{lem:mollifier} and Proposition \ref{prop_liminf}, we find
\begin{equation*}
\begin{split}
\int_{\mathbb{R}^N}|\nabla u(x)|^p\,dx&\geq \limsup\limits_{\delta\to0^+}\frac{1}{G(\delta^{1-s})}\int_{B(x,\delta)}G\left(\frac{|u(x)-u(y)|}{|x-y|^{s}}\right)\frac{dy}{|x-y|^N}\\
&\geq\liminf\limits_{\delta\to0^+}\frac{1}{G(\delta^{1-s})}\int_{B(x,\delta)}G\left(\frac{|u(x)-u(y)|}{|x-y|^{s}}\right)\frac{dy}{|x-y|^N}\\
&\geq\liminf\limits_{\delta\to0^+}\frac{1}{G(\delta^{1-s})}\int_{B(x,\delta)}G\left(\frac{|u_k(x)-u_k(y)|}{|x-y|^{s}}\right)\frac{dy}{|x-y|^N}\\
&\geq\int_{\mathbb{R}^N}|\nabla u_k(x)|^p\,dx.
\end{split}
\end{equation*}
Because of \eqref{conv_ae2} and Fatou's Lemma, we conclude
\begin{equation*}
\liminf\limits_{k\to\infty}\int_{\mathbb{R}^N}|\nabla u_k(x)|^p\,dx\geq\int_{\mathbb{R}^N}|\nabla u(x)|^p\,dx.
\end{equation*}
\end{proof}
We finish this section by proving Theorem \ref{Th2}.
\begin{proof}[Proof of  Theorem \ref{Th2}]
Let $u\in L^G(\mathbb{R}^N)$ and, for $k\in\mathbb{N}$, let $u_k$ be the function defined as in \eqref{mollified} with $r=\frac{1}{k}$. Then, because of \eqref{hyp_inf} and Lemma \ref{lem:mollifier}, there exists $c>0$ such that
\begin{equation*}
\liminf\limits_{\delta\to 0^+}\frac{1}{G(\delta^{1-s})}\int_{\mathbb{R}^N}\int_{B(x,\delta)}G\left(\frac{|u_k(x)-u_k(y)|}{|x-y|^{s}}\right)\frac{dydx}{|x-y|^N}<c,\qquad\text{for }k\in\mathbb{N}.
\end{equation*}
Thus, by Proposition \ref{prop_liminf},
\begin{equation*}
\int_{\mathbb{R}^N}|\nabla u_k(x)|^pdx\leq c\qquad\text{for }k\in\mathbb{N}.
\end{equation*}
Since $u\in L^G(\mathbb{R}^N)$ and, by \eqref{hypotheses},
\begin{equation*}
\int_{\mathbb{R}^N}|u_k(x)|^p\, dx\leq c_\ell\int_{\mathbb{R}^N}G(|u(x)|)dx\qquad\text{for }k\in\mathbb{N},
\end{equation*}
the sequence $\{u_k\}_{k\in\mathbb{N}}$ is bounded in $W^{1,p}(\mathbb{R}^N)$ and then, there exists $v\in W^{1,p}(\mathbb{R}^N)$ and a subsequence (still denoted by $\{u_k\}_{k\in\mathbb{N}}$) such that
$ u_k\rightharpoonup v$ in $W^{1,p}(\mathbb{R}^N)$. Since $u\in L^G(\mathbb{R}^N)$, in particular $u\in L^1_{loc}(\mathbb{R}^N)$ so that, by \eqref{conv_ae}, we conclude $v=u$ and, thus, $u\in W^{1,p}(\mathbb{R}^N)$.
\end{proof}
\section{Further comments}\label{Section:Further}

\subsection{The case of a sequence}\hfill\newline
Let us begin by recalling the definition of $\Gamma$-convergence of functionals.
\begin{definition}
Let $X$ be a metric space and $F,F_k:X\mapsto\overline{\mathbb{R}}$. We say that $F_k$ $\Gamma$-converges to F if, for every $u\in X$ the following conditions hold:
\begin{enumerate}
\item (liminf--inequality) For every sequence $\{u_k\}_{k\in\mathbb{N}}\subset X$ such that $u_k\to u$ in $X$,
\begin{equation*}
F(u)\leq\liminf\limits_{k\to\infty} F_k(u_k).
\end{equation*}
\item (limsup--inequality) For every $u\in X$, there is a sequence $\{u_k\}_{k\in\mathbb{N}}\subset X$ such that $u_k\to u$ in $X$ and
\begin{equation*}
\limsup\limits_{k\to\infty} F_k(u_k)\leq F(u).
\end{equation*}
\end{enumerate}
\end{definition}
A full $\Gamma$-convergence result for functionals also requires the compactness of bounded sequences. In our setting this reads as
\begin{proposition}\label{compact}
Let $\{u_k\}_{k\in\mathbb{N}}\subset L^G(\mathbb{R}^N)$. If 
\begin{equation*}
\sup\limits_{k\in\mathbb{N}} \frac{1}{G(\delta_k^{1-s})}\Psi_{s,G,\delta_k}(u_k)<+\infty\qquad\text{and}\qquad\sup\limits_{k\in\mathbb{N}} \Phi_G(u_k)<+\infty,
\end{equation*} 
then $u_k\to u$ in $L_{loc}^G(\mathbb{R}^N)$ for some $u\in W^{1,p}(\mathbb{R}^N)$.
\end{proposition}
Taking in mind the results proven in the former sections, it is clear that a full $\Gamma$-convergence result follows once we prove the liminf--inequality and Proposition \ref{compact}.  The first goal is achieved combining Proposition \ref{mainprop1} and Lemma \ref{lem:mollifier}.
\begin{proposition}
Let $\delta_k\to 0$ and $\{u_k\}_{k\in\mathbb{N}}\subset L^G(\mathbb{R}^N)$ be a sequence such that $u_k\to u$ in $L^G(\mathbb{R}^N)$ for some $u\in W^{1,G}(\mathbb{R}^N)$ as $k\to+\infty$. Then,
\begin{equation*}
\int_{\mathbb{R}^N}|\nabla u(x)|^pdx\leq\liminf\limits_{k\to+\infty}\,\frac{p(1-s)}{G(\delta_k^{1-s})}\Psi_{s,G,\delta_k}(u_k).
\end{equation*}
\end{proposition}
\begin{proof}
The proof follows as that of \cite[Theorem 1-D2)]{Bellido2015}. For each $k\in\mathbb{N}$ and $r>0$, let $u_{k,r}$ and $u_r$ be defined as in \eqref{mollified}, the regularized functions of $u_k$ and $u$ respectively. By Proposition \ref{mainprop1} and Lemma \ref{lem:mollifier},
\begin{equation*}
\int_{\mathbb{R}^N}|\nabla u_r(x)|^pdx=\lim\limits_{k\to+\infty}\frac{p(1-s)}{G(\delta_k^{1-s})}\Psi_{s,G,\delta_k}(u_{k,r})\leq\liminf\limits_{k\to+\infty}\frac{p(1-s)}{G(\delta_k^{1-s})}\Psi_{s,G,\delta_k}(u_{k}).
\end{equation*}
We conclude since, by standard properties of mollifiers, $u_r\to u$ in $W^{1,p}(\mathbb{R}^N)$ as $r\to0$. 
\end{proof}
Regarding Proposition \ref{compact}, let us note the following. When analyzing the limit $s\to1^+$ of the complete functional
\begin{equation*}
\Psi_{s,G,\infty}(u)=\int_{\mathbb{R}^N}\int_{\mathbb{R}^N} G\left(\frac{|u(x)-u(y)|}{|x-y|^{s}}\right)\frac{dydx}{|x-y|^N},
\end{equation*}
the liminf--inequality and the compactness are proved with the aid of the following inequality (cf. \cite[Theorem 5.2]{FernandezBonder2019}),
\begin{equation}\label{ineqSalort}
(1-s_1)\Psi_{s_1,G,\infty}(u)\leq 2^{1-s_1}(1-s_2)\Psi_{s_2,G,\infty}(u)+\frac{2\mathfrak{c}\omega_N(1-s_1)}{s_1}\Phi_G(u).
\end{equation}
The use of this inequality is twofold. At one hand, it allows to prove the liminf--inequality by fixing $s_1=t$ for some $0<t<1$, and then performing the limit $s_2=s_k\to1^-$ as $k\to+\infty$ and concluding by taking the limit $t\to1^-$. In other words, inequality \eqref{ineqSalort} allows to bound $(1-s)\Psi_{s,G,\infty}$ at $s_1$ using two terms: one at $s_2$ that does not depend on $s_1$ and other that vanishes as $s_1\to1^-$. 

On the other hand, let $\{u_k\}_{k\in\mathbb{N}}$ be such that 
\begin{equation*}
\sup\limits_{k\in\mathbb{N}}(1-s_k)\Psi_{s_k,G,\infty}(u_k)<+\infty\qquad\text{and}\qquad \sup\limits_{k\in\mathbb{N}}\Phi_{G}(u_k)<+\infty.
\end{equation*}
Then, take $0<t<1$ be fixed, so that, by \eqref{ineqSalort}, we have $\{u_k\}_{k\in\mathbb{N}}\subset W^{t,G}(\mathbb{R}^N)$ is bounded and hence, by the compactness provided by \cite[Theorem 3.1]{FernandezBonder2019}, there  exists $u\in L^G(\mathbb{R}^N)$ such that $u_k\to u$ in $L_{loc}^G(\mathbb{R}^N)$.

Adapting the approach of \cite[Theorem 5.2]{FernandezBonder2019} to the case $\delta\to0^+$ yields the following.
 \begin{lemma}\label{useles}
 Let $u\in L^G(\mathbb{R}^N)$ and $0<\delta_2<\delta_1$. Then
 \begin{equation*}
 \frac{1}{G(\delta_1^{1-s})}\Psi_{s,G,\delta_1}(u)\leq\frac{1}{G(\delta_2^{1-s})}\Psi_{s,G,\delta_2}(u)+\frac{\mathfrak{c}\,\omega_N N}{s}\left(1-\left(\frac{\delta_2}{\delta_1}\right)^s\right)\frac{1}{G(\delta_2^{1-s})}\Phi_G\left(u\,\delta_2^{-s}\right).
 \end{equation*}
 \end{lemma}
Unfortunately, contrary to the case $s\to1^-$, for which the function is integrated over the whole $\mathbb{R}^N$ independently of $0<s<1$ (so we have a complete picture of $u$ at every step of the limit), by the very definition of the problem as $\delta\to0^+$, we can not drop the dependence on $\delta_2$ of the second term in Lemma \ref{useles}. That is, given $\delta_k\to0$ and $\{u_k\}_{k\in\mathbb{N}}$ such that 
\begin{equation*}
\sup\limits_{k\in\mathbb{N}}\frac{1}{G(\delta_k^{1-s})}\Psi_{s,G,\delta_k}(u_k)<+\infty\qquad\text{and}\qquad \sup\limits_{k\in\mathbb{N}}\Phi_{G}(u_k)<+\infty,
\end{equation*}
Lemma \ref{useles} does not allow us to conclude 
\begin{equation}\label{goal}
\sup\limits_{k\in\mathbb{N}}\Psi_{s,G,\delta_0}(u_k)<+\infty,
\end{equation}
for some $\delta_0>0$ fixed. Observe that, if \eqref{goal} holds, then $\{u_k\}_{k\in\mathbb{N}}\subset W^{s,G,\delta_0}(\mathbb{R}^N)$ is bounded and, by Theorem \ref{teo:compact}, we get the existence of a convergent subsequence. We left this as an open problem.
\subsection{Behavior of the spectrum as $\delta\to0^+$}\hfill\newline
The spectrum of the nonlocal operator associated to the functional
\begin{equation*}
\Psi_{s,G,\infty}(u)=\int_{\mathbb{R}^N}\int_{\mathbb{R}^N} G\left(\frac{|u(x)-u(y)|}{|x-y|^{s}}\right)\frac{dydx}{|x-y|^N},
\end{equation*}
is analyzed in \cite{Salort2020}. In particular, given $\Omega\subset\mathbb{R}^N$ and $\lambda\in\mathbb{R}$, the author analyzed the problem
\begin{equation}\label{eigenprob}
        \left\{
        \begin{tabular}{ll}
        $(-\Delta_g)_{\infty}^su=\lambda g(|u|)\frac{u}{|u|}$, &in $\Omega$, \\
        $\mkern+65mu u=0$ &on $\mathbb{R}^N\backslash\Omega$,
        \end{tabular}
        \right.
\end{equation}
where $(-\Delta_g)_{\infty}^s$ denotes the \textit{fractional $g\,$-Laplacian},
\begin{equation*}
(-\Delta_g)^s u(x)= (-\Delta_g)_{\infty}^s u(x)=p.v. \int_{\mathbb{R}^N}g(|D_su|)\frac{D_su}{|D_su|}\frac{dy}{|x-y|^N},
\end{equation*}
with $g=G'$ and 
\begin{equation*}
D_su=D_su(x,y)=\frac{u(x)-u(y)}{|x-y|^s}.
\end{equation*}
The operator $(-\Delta_g)^s$ can be seen as the nonlocal counterpart of the $g$-Laplacian, namely
\begin{equation*}
-\Delta_g u=-div\left(g(|\nabla u|)\frac{\nabla u}{|\nabla u|}\right).
\end{equation*}
Problem \eqref{eigenprob} is the Euler-Lagrange equation corresponding to the minimization problem
\begin{equation*}
\alpha_{1,\mu}=\inf\limits_{u\in M_\mu}\frac{\Psi_{s,G,\infty}(u)}{\Phi_G(u)}\qquad\text{with}\qquad M_\mu=\{u\in W_0^{s,G,\infty}(\Omega): \Phi_g(u)=\mu\},
\end{equation*}
where $W_0^{s,G,\infty}(\Omega)=\{u\in W^{s,G,\infty}(\Omega):\ u=0\text{ on }\mathbb{R}^N\backslash\Omega\}$. This construction was latter extended (cf. \cite{Bahrouni2022}) to provide an increasing sequence $\{\lambda_k^{s,G,\infty}\}_{k\in\mathbb{N}}$ of variational eigenvalues. Since the $\Gamma$-convergence of functionals implies the convergence of the minima of those functionals (cf. \cite{Braides2002}), by the results of \cite{FernandezBonder2019}, the author proves (cf. \cite[Proposition 6.3]{Salort2020}),
\begin{equation*}
\lim\limits_{s\to1^-}(1-s)\alpha_{1,\mu,s}=\alpha_{1,\mu,1}.
\end{equation*}
Similar results for the case $G(t)=t^p$ were proven in \cite{Brasco2016}.

It is natural then to set the same question when dealing with the truncated version of the fractional $g$-Laplacian, namely, given and horizon $\delta>0$, $s\in(0,1)$, $\Omega\subset\mathbb{R}^N$ and $\lambda\in\mathbb{R}$, consider the problem 
\begin{equation*}
        \left\{
        \begin{tabular}{ll}
        $(-\Delta_g)_{\delta}^su=\lambda g(|u|)\frac{u}{|u|}$, &in $\Omega$, \\
        $\mkern+59.5mu u=0$ &on $\partial_\delta\Omega$,
        \end{tabular}
        \right.
\end{equation*}
where $(-\Delta_g)_{\delta}^s$ denotes the {\it peridynamic} fractional $g$-Laplacian,
\begin{equation*}
(-\Delta_g)_{\delta}^s u(x)= p.v. \int_{B(x,\delta)}g(|D_su|)\frac{D_su}{|D_su|}\frac{dy}{|x-y|^N}
\end{equation*} 
and $\partial\Omega_{\delta}=\{z\in\mathbb{R}^N\backslash\Omega:|x-z|<\delta,\text{ for }x\in\Omega\}$. The case $G(t)=t^p$ was addressed in \cite{Bellido2021, Bellido2021a}, where it was proven
\begin{equation*}
\frac{1}{\delta^{p(1-s)}}\lambda_k^{\delta,s,p}\to K_{N,p}\,\lambda_k^{1,p}\qquad\text{as }\delta\to 0^+,\text{ for }k\in\mathbb{N}.
\end{equation*}
Based on the above results, once a full $\Gamma$-convergence was proved, it is to be expected that the following holds, 
\begin{equation*}
\lim\limits_{\delta\to0^+}\frac{p(1-s)}{G(\delta^{1-s})}\lambda_1^{\delta,s,G}=K_{N,p}\lambda_1^{1,p},
\end{equation*}
for $G$ an Orlicz function satisfying \eqref{hypotheses} and $p=index(G)$.


\end{document}